\documentclass[11pt]{article}
\usepackage[utf8]{inputenc}
\usepackage{amsmath}
\usepackage{amsfonts}
\usepackage{amssymb}
\usepackage{amsthm}
\usepackage[colorlinks=true,linkcolor=blue,citecolor=blue]{hyperref}
\usepackage{enumerate}
\usepackage{framed}
\usepackage{tikz}
\usepackage{geometry}

\newtheorem{lemma}{Lemma}[section]

\newtheorem{theorem}[lemma]{Theorem}
\newtheorem{fact}[lemma]{Fact}
\newtheorem{corollary}[lemma]{Corollary}
\newtheorem{conj}[lemma]{Conjecture}
\theoremstyle{remark}
\newtheorem{Bemerkung}[lemma]{Remark} 
\newtheorem{Beispiel}[lemma]{Example} 

\newenvironment{remark}{\begin{Bemerkung}}{\qede\end{Bemerkung}}
\newenvironment{ex}{\begin{Beispiel}}{\qede\end{Beispiel}}

\renewcommand{\H}{\mathcal{H}}
\newcommand{\Fix}{\operatorname{Fix}}
\newcommand{\cl}{\operatorname{cl}}
\newcommand{\conv}{\operatorname{conv}}
\newcommand\argmin{\mathop{\rm argmin}}
\newcommand{\qede}{\hspace*{\fill}$\Diamond$\medskip}

\begin{document}

\title{Norm Convergence of Realistic Projection and
       Reflection Methods}

\author{Jonathan M. Borwein$^{a}$ \quad Brailey Sims$^{a}$ \quad Matthew K. Tam$^{a}$
        \thanks{Corresponding author. Email: \href{mailto:matthew.tam@uon.edu.au}{matthew.tam@uon.edu.au}}\\ \vspace{6pt}
        $^{a}${\em CARMA Center, University of Newcastle, Callaghan, NSW 2308, Australia} }

\maketitle

\begin{abstract}
 We provide sufficient conditions for norm convergence of various projection and reflection methods, as well as giving limiting examples regarding convergence rates.
\end{abstract}



\section{Introduction}
The \emph{(2-set) convex feasibility problem} asks for a point contained within the intersection of two closed convex sets of a Hilbert space. \emph{Projection and reflection methods} represent a class of algorithmic schemes which are commonly used to solve this problem. Some notable projection and reflection methods include the \emph{method of alternating projections}, the \emph{Douglas--Rachford method}, the \emph{cyclic Douglas--Rachford scheme}, and of course many extensions and variants. For details see \cite{ER11,BaB93,BaB96,LM67,BaCL04,cycDR,cycDRinfeas}, and the references therein. Each iteration of these methods, employes some combination of \emph{(nearest point) projections} onto the constraint sets. Their sustained popularity, even in settings without convexity, is due to their relative simplicity and ease-of-implementation, in addition to observed good performance \cite{ERT07,GE08,DRcomb,DRmatrix}.

For the majority of projection and reflection methods applied to general closed convex sets only weak convergence of the iterates can be guaranteed. . Hundal, relatively recently \cite{HH04}, gave the first explicit example of an alternating projection iteration which does not converge in norm. A number of variants and extensions to this example have since been published \cite{MR03,K08,BBDHV05}, some of which cover the case of non-intersecting sets (infeasible  problems). These examples consider two sets, the first being either a closed subspace of finite codimension or one of its half-spaces, and the second a convex cone ``built-up" from three dimensional ``building blocks". For non-convex sets, the question of convergence is more difficult, and currently result focus on the finite dimensional setting \cite{LLM09,BLPW13,HL13,HLN13,BaN14}.

 In light of these examples, it is natural ask what compatibility conditions on the two sets are required to ensure norm convergence. This is further motivated by the pleasing physical interpretation of norm convergence as the ``error" becoming arbitrarily small \cite{BaC01}.

When the constraint sets satisfies certain regularity properties, norm convergence of the method of alternating projections can be guaranteed \cite{BaB93,BaB96}, and in some cases a linear rate of converge can also be assured. These regularity conditions are most easily invoked in the analysis of the method of alternating projections. This is because each iteration of the method produces a point contained within one of the two constraint sets for which the regularity properties can be invoked. On the other hand, the Douglas--Rachford method generates points that need not lie within the sets, making it more difficult to analyze. Consequently less is known of its behaviour.  Further, to the authors' knowledge no explicit Hundal-like counter-example is known for the Douglas--Rachford algorithm. For recent progress, on convex Douglas--Rachford methods see \cite{HLN13,BaBPW13}.

An important practicable instance of the feasibility problem occurs when the space is a Hilbert lattice, one of the sets  is the positive Hilbert cone, and the other is a closed affine subspace with finite codimension. Problems of this kind arise, for example, in the so called `moment problem' (see \cite{BaB93}). Applied to this type of feasibility problem, Bauschke and Borwein proved that the method of alternating projection converges in norm whenever the affine subspace has codimension one \cite{BaB93}. The same was conjectured to stay true for any finite codimension, but remains a stubbornly open problem.

\begin{quote}
 The goal of this paper is two-fold. First, to formulate unified sufficient conditions for norm convergence of fundamental projection and reflection methods when applied to feasibility problems with finite codimensional affine space and convex cone constraints, and second, to give examples and counter-examples regarding the convergence rate of these methods.
\end{quote}

The remainder of the paper is organized as followed: in Section~\ref{sec:preliminaries} we recall definitions and important theory for our analysis; in Section~\ref{sec:sufficient} we formulate sufficient conditions for norm convergence, which we then specialize to Hilbert cones. Finally, in Section~\ref{sec:rates} we give various examples and counter-examples regarding the rate of converge, and the interplay with regularity of the constraints sets, for both projection and reflection methods.

\section{Preliminaries}\label{sec:preliminaries}
Throughout, we assume that $\mathcal H$ is a real Hilbert space equipped with inner product $\langle\cdot,\cdot\rangle$ and induced norm $\|\cdot\|$. We denote the \emph{range} (resp. \emph{nullspace}) of the a mapping $T$ by $R(T)$ (resp. $N(T)$).

The \emph{(nearest point) projection} onto a set $S\subseteq\mathcal H$ is the mapping $P_S:\mathcal H\to S$ given by
 $$P_Sx:=\argmin_{s\in S}\|x-s\|.$$
If $S$ is closed and convex, then $P_S$ is well defined and has the characterization
 \begin{equation}
  \langle x-P_Sx,S-P_Sx\rangle\leq 0.\label{eq:varChar}
 \end{equation}
The \emph{reflection} with respect to $S$ is the mapping $R_S:\mathcal H\to \mathcal H$ defined by $R_S:=2P_S-I$.

Recall that a \emph{cone} is a set $K\subseteq\mathcal H$ such that $\mathbb R_+K\subseteq K$. A cone $K$ is \emph{pointed} if $K\cap(-K)=\{0\}$, \emph{generating} if $K-K=\mathcal H$,
and \emph{(norm) normal} if there exist a (norm) neighbourhood basis, $\mathcal V$, of $0$ such that
 $$V=(V+K)\cap (V-K) \text{ for all }V\in\mathcal V.$$
Given a set $S\subseteq\mathcal H$, its \emph{negative polar cone} is the convex cone
  $$S^\ominus:=\{x\in\mathcal H:\langle x,S\rangle\leq 0\}.$$
If $S$ is nonempty, $(S^\ominus)^\ominus=  \cl\conv(\mathbb R_+S)$ (see, for example, \cite{BV10}). In particular, if $K$ is a closed convex cone then $(K^\ominus)^\ominus=K$. The \emph{positive polar cone} to $S$ is defined similarly and $S^\oplus:=-S^\ominus$.

We have the following useful conic duality results.
\begin{fact}
 Let $X$ be a Banach space, and $K\subseteq X$ be a closed convex cone. Then:
  \begin{enumerate}[(a)]
   \item $K$ is pointed if and only if $K^\ominus-K^\ominus$ is weak-star dense in $X^*$.
   \item $K^\ominus$ is pointed if and only if $K-K$ is weakly dense in $X$.
   \item $K^\ominus$ is normal if and only if $K$ is generating.
  \end{enumerate}
\end{fact}
\begin{proof}
 See, for example, \cite[Th.~2.13 \& Th.~2.40]{AT07}.
\end{proof}

A Hilbert space can be expressed as the direct sum of any closed subspace and its orthogonal complement. The following theorem is a fine analogue for closed convex cones.

\begin{theorem}[Moreau decomposition theorem] \label{th:moreau}
 Suppose $K\subseteq\mathcal H$ is a nonempty closed convex cone.
 For any $x\in\mathcal H$,
  \begin{enumerate}[(a)]
   \item $x=P_Kx+P_{K^\ominus}x$.
   \item $\langle P_Kx,P_{K^\ominus}x\rangle=0$.
   \item $\|x\|^2=d_K^2(x)+d_{K^\ominus}^2(x).$
  \end{enumerate}
\end{theorem}
\begin{proof}
 See, for example, \cite[Th.~6.29]{BaC11}. For extensions see \cite{CR11}.
\end{proof}

Let $X$ be a (real) linear space. Recall that a \emph{partially ordered linear space} is a pair $(X,K)$ where $K\subseteq X$ is a convex pointed cone and the ordering $\leq_K$ on $X$ induced by $K$ is
 $$x\leq_K y \iff y-x\in K.$$
In addition, if the ordering defines a lattice we say $(X,K)$ is a \emph{linear lattice}. In this case, the supremum (resp. infimum) of the doubleton $\{x,y\}\subseteq X$ is denoted by $x\vee y$ (resp. $x\wedge y$). The \emph{positive part}, \emph{negative part} and \emph{modulus} of a point $x\in\mathcal H$ are given by $x^+:=x\vee 0,x^-:=(-x)\vee 0$ and $|x|:=x\vee(-x)$, respectively. 

A \emph{normed lattice} is a linear lattice $(X,K)$ with a norm such that
 $$|x|\leq_K|y|\implies \|x\|\leq\|y\|.$$
A \emph{Banach lattice} is a complete normed lattice, and a \emph{Hilbert lattice} a Banach lattice in which the norm arises from an inner product. In a Hilbert lattice $(\mathcal  H,K)$ the cone $K$ is characterised by (see, for example, \cite[Th.~8]{BY84})
  \begin{equation}
    K=K^\oplus=(-K^\ominus)=\{x\in\mathcal H:\langle x,K\rangle\geq 0\}.\label{eq:HilbertCone}
  \end{equation}
Where there is no ambiguity, we will say that $X$ is a linear/Banach/Hilbert lattice (i.e., without reference to the cone) and denote the order cone by $X^+$.

\begin{fact}[Basic properties of linear lattices]
 Let $(X,K)$ be a (real) linear lattice and $x,y\in X$. Then
  \begin{enumerate}[(a)]
    \item $(x\vee y)+(x\wedge y)=x+y$.
    \item $(-x)\wedge(-y)=-(x\vee y)$.
    \item $(x+y)^+\leq_K x^++y^+$
    \item $x=x^+-x^-$ and $|x|=x^++x^-$
  \end{enumerate}
 Further, when $(X,K)$ is a Hilbert lattice, $P_K^+=x^+$.
\end{fact}
\begin{proof}See, for example, \cite{P67} and \cite{BaB93}. \end{proof}

\begin{remark}
 In texts on ordered topological vectors spaces, it is common (but not uniformly so) to define a ``cone" to be both pointed and convex, in addition to being closed under positive scalar multiplication.
\end{remark}

The following fact allows one to exploit the order structure induced by a closed convex pointed cone. We require and so state only the simplest reflexive results.

\begin{fact}[Normal cones in reflexive space]\label{fact:daniell}
 Let $X$ be a reflexive Banach space, and $K\subseteq X$ a closed convex pointed cone. The following are equivalent.
  \begin{enumerate}[(a)]
   \item If $(x_n)_{n=1}^\infty\subseteq X$ with $x_n\leq_K x_{n+1}$ and $\sup_n\|x_n\|<\infty$, then $(x_n)_{n=1}^\infty$ is norm convergent.
   \item If $(x_n)_{n=1}^\infty\subseteq X$ with $x_n\leq_K x_{n+1}$ and there exists $x\in X$ such that $x_n\leq_K x$, then $(x_n)_{n=1}^\infty$ is norm convergent.
   \item $K$ is (norm) normal.
  \end{enumerate}
\end{fact}
\begin{proof}
 See, for example, \cite[Th.~2.45]{AT07}.
\end{proof}

\section{Sufficient Conditions for Norm Convergence}\label{sec:sufficient}

Suppose we have two sequences  $(\lambda_n)_{n=1}^\infty\subseteq\H$, and $(\kappa_n)_{n=1}^\infty\subseteq K\subseteq\H$, for some closed convex cone $K$. Given an initial point $x_0\in\H$, iteratively define the sequence $(x_n)_{n=1}^\infty$ by
\begin{equation}\label{eq:iterSch}
 x_{n}:=x_{n-1}-\kappa_{n}+Q\lambda_{n},
\end{equation}
where $Q:\H\to M$ is a linear mapping, and $M$ is a finite dimensional subspace of $\H$. Using the linearity of $Q$, \eqref{eq:iterSch} implies
\begin{equation}\label{eq:iterTele}
 x_{n}-x_0=-\sigma_{n}+Q\alpha_{n},
\end{equation}
where
 $$\sigma_{n}=\sum_{k=1}^{n}\kappa_k\in K,\quad \alpha_n:=\sum_{k=1}^{n}\lambda_n.$$

\begin{framed}
 Henceforth, unless explicitly stated otherwise, $(x_n)_{n=1}^\infty$ will denote a sequence of the form given in \eqref{eq:iterSch}.
\end{framed}

We now give two important examples, Examples \ref{ex:DR} and \ref{ex:VN}, of sequences satisfying the above assumptions. In both, we suppose that $S$ is a closed convex cone, and that $A$ is a closed affine subspace of finite codimension. Later, in Example~\ref{ex:relaxedDR}, we supply a unified extension.

In what follows, we denote by $Q$ the projection onto the (finite dimensional) orthogonal complement of the subspace parallel to $A$, so that (see Remark~\ref{remark:Q})
  $$P_Ax=x+Q(\overline x-x), \mbox{~for any~} \overline x\in A.$$

\begin{ex}[Douglas--Rachford sequences]\label{ex:DR}
 For any $x_0\in\H$ the \emph{Douglas--Rachford sequence} is defined by
  $$x_{n+1}:=T_{S,A}x_n\text{ where } T_{S,A}:=\frac{I+R_AR_S}{2},$$
 which, for any $\overline x_n\in A$, is expressible as
  \begin{align}
   x_{n+1} &=P_Sx_n+Q(\overline x_n- R_Sx_n) \notag{}\\
  &=x_n-P_{S^\ominus}x_n+Q(\overline x_n- R_Sx_n). \label{eq:DRiter}
  \end{align}
 So in this case, $K:=S^\ominus, \kappa_{n+1}=P_{S^\ominus}x_n$ and $\lambda_{n+1}=\overline x_n-R_{S}x_n$.
\end{ex}

\begin{ex}[von Neumann sequences]\label{ex:VN}
 For any $x_0\in\H$, the \emph{von Neumann sequence} is defined by
  $$x_{n+1}:=P_AP_Sx_n,$$
 which, for any $\overline x_n\in A$, is expressible as
 \begin{align*}
   x_{n+1}&=P_Sx_n+Q(\overline x_n-P_Sx_n),\\
          &= x_n-P_{S^\ominus}x_n+Q(\overline x_n-P_Sx_n).
 \end{align*}
 So here, again $K:=S^\ominus$ and $\kappa_{n+1}=P_{S^\ominus}x_n$ while $\lambda_{n+1}=\overline x_n-P_Sx_n$.
\end{ex}

\begin{remark}[Further properties]
 Whenever $S\cap A\neq\emptyset$, the Douglas--Rachford (resp. von Neumann) sequence converges weakly to a point in $\Fix T_{S,A}$ (resp. $S\cap A$), see, for example, \cite{BaC11}), and is \emph{Fej\'er monotone} with respect to $\Fix T_{S,A}$ (resp. $S\cap A$). Consequently, the sequence is bounded, and is norm convergent whenever it contains a norm convergent subsequence.
\end{remark}

 \begin{remark}[Computation of $Q$]\label{remark:Q}
 Let $\lambda\in\mathbb{R}^N$. As in \cite[Section~5]{BaB93}, define
 $$S :=  \H^+,\quad A := T^{-1}\lambda,$$
where $T:\H\to\mathbb R^N$ is a linear, continuous and given by $x\mapsto(\langle t_i,x\rangle)_{i=1}^N$ for given linearly independent vectors $t_i\in\H$.
Letting $Q:=T^\ast(TT^\ast)^{-1}$, we have as above
 $$P_Ax=x+Q(\overline x-x),\quad\text{for any }\overline x\in A.$$
Whence,
 $$R_Ax = x+2Q(\overline x-x),\quad R_Sx = 2x^+-x = |x|.$$\end{remark}

We give one further example, although many other variants are also possible. It includes both the von Neumann and Douglas--Rachford sequences as special cases.

\begin{ex}[Relaxed Douglas--Rachford sequences]\label{ex:relaxedDR}
For any $x_0\in\H$, consider the relaxation of the Douglas--Rachford sequence given by
 $$x_{n+1}:=T^c_{S,A}x_n$$
where
 $$T^c_{S,A} := c I + (1-c)R_A^b R_S^a,\quad
 R^a_K           := a I + (1-a)R_S,\quad
 R^b_A           := b I + (1-b)R_A,$$
for some $a,b\in{[0,1[}$ and $c\in{[0,1[}$.

That is, we replace each of $T_{S,A},R_S$ and $R_A$ in the Douglas--Rachford method, with a convex combination of itself and the identity.
When $a=b=0$ and $c=1/2$ we recover the Douglas-Rachford iteration, and when $a=b=1/2$ and $c=0$ we obtain the von Neumann iteration.

For any $\overline x_n\in A$, it is expressible as
\begin{align}
 x_{n+1}
 &= x_n+(1-c)\left(-x_n+R_A^b R_S^a x_n\right) \notag \\
 &= x_n+(1-c)\left(-x_n+b R_S^a x_n+(1-b)\left[2P_AR_S^a x_n-R_S^ax_n\right]\right) \notag \\
 &= x_n+(1-c)\left(-x_n+b R_S^a x_n-(1-b)R_S^ax_n+2(1-b)P_AR_S^a x_n\right) \notag \\
 &= x_n+(1-c)\left(-x_n+(2b-1) R_S^a x_n+2(1-b)\left[R_S^a x_n+Q(\overline x_n-R_S^a x_n)\right]\right) \notag \\
 &= x_n+(1-c)\left(-x_n+R_S^a x_n+2(1-b)Q(\overline x_n-R_S^a x_n)\right) \notag \\
 &= x_n+(1-c)\left(-x_n+a x_n+(1-a)R_S x_n+2(1-b)Q(\overline x_n-R_S^a x_n)\right) \notag \\
 &= x_n+(1-c)\left(-2(1-a)P_{S^\ominus} x_n+2(1-b)Q(\overline x_n-R_S^a x_n)\right) \notag \\
 &= x_n-2(1-a)(1-c)P_{S^\ominus} x_n+2(1-b)(1-c)Q(\overline x_n-R_S^a x_n). \label{eq:relaxDRiter}
\end{align}

That is, $K:=S^\ominus$,
 $$\kappa_{n+1}=2(1-a)(1-c)P_{S^\ominus}x_n,\quad \lambda_{n+1}=2(1-b)(1-c)(\overline x_n-R_S^a x_n).$$

Both $R_A^b$ and $R_S^a$ are \emph{averaged operators} (for a definition and more, see \cite{BaC11}). As the composition of averaged operators, $R_A^b R_S^a$ is averaged, and thus $T^c_{S,A}$ is averaged. Furthermore,
$\Fix T_{S,A}\neq\emptyset$ whenever $S\cap A\neq\emptyset$.
We may invoke \cite[Pr.~5.15]{BaC11} to see that the sequence $(x_n)_{n=1}^\infty$ is weakly convergent to a point in
$\Fix T^c_{S,A}$, and Fej\'er monotone w.r.t. $\Fix T^c_{S,A}$. In particular, the latter implies that the sequence is bounded, and is norm convergent whenever it contains a norm convergent subsequence.
\end{ex}

The following lemma gives some insight into what might cause the sequence $(x_n)_{n=1}^\infty$ to fail to converge in norm.

\begin{lemma}[Recession directions]\label{lem:cluster}
 Let $K\subseteq\H$ be a nonempty closed convex pointed norm normal cone.
 Suppose $(x_n)_{n=1}^\infty$ is bounded sequence of the form given in \eqref{eq:iterSch}.
 Then, either $(x_n)_{n=1}^\infty$ contains a norm convergent
 subsequence or the set of norm cluster points of $(Q\alpha_n/\|Q\alpha_n\|)_{\{n
:Q\alpha_n\neq 0\}}$ is nonempty and contained in $K$. In particular, the latter implies
$ R(Q)\cap K \neq \{0\}$.
\end{lemma}
\begin{proof}
 Since $(x_n)_{n=1}^\infty$ is bounded, by \eqref{eq:iterTele}, we see that $(\sigma_n)_{n=1}^\infty$ is bounded if and only if
 $(Q\alpha_n)_{n=1}^\infty$ is bounded. We distinguish two cases: (i) $(Q\alpha_n)_{n=1}^\infty$ contains a bounded
 subsequence, say $(Q\alpha_{n_k})_{k=1}^\infty$, or (ii) no  subsequence of $(Q\alpha_n)_{n=1}^\infty$ is bounded.

 (i) In this case, by passing to a further subsequence if
 necessary, we may assume that $(Q\alpha_{n_k})_{k=1}^\infty$ converges weakly and hence in norm since it is contained within
 a finite dimensional subspace. Further, $(\sigma_{n_k})_{k=1}^\infty$ is bounded,
 and, along with $\sigma_n$ itself, increasing with respect to the partial order induced by $K$, so it converges in norm (by Fact~\ref{fact:daniell}).
 Equation \eqref{eq:iterTele} now implies  that $(x_{n_k})_{k=1}^\infty$ converges in
 norm.

(ii) Let $q_n:=Q\alpha_n/\|Q\alpha_n\|$ when $\|Q\alpha_n\|\neq 0$.
And, let $q$ be an arbitrary norm cluster point of $(q_n)_{\{n
:Q\alpha_n\not= 0\}}$, which exists because $(q_n)_{\{n
:Q\alpha_n\neq 0\}}$ is bounded and contained within a finite
dimensional subspace. Let $(q_{n_k})_{k=1}^\infty$ be a subsequence
convergent to $q$, which by passing to a further subsequence if
necessary, we may assume has $0 < \|Q\alpha_{n_k}\|\to+\infty$.
Then,
  $$\frac{x_{n_k}-x_0}{\|Q\alpha_{n_k}\|}=\frac{-\sigma_{n_k}}{\|Q\alpha_{n_k}\|}+q_{n_k}
  \implies q=\lim_{k\to\infty}\frac{\sigma_{n_k}}{\|Q\alpha_{n_k}\|}.$$
This completes the proof.
\end{proof}

\begin{remark}
 If $S\subseteq\H$ is a closed convex generating cone, then $S^\ominus$ is a closed convex pointed norm normal cone (see, for example, \cite[Cor.~2.43]{AT07}), so
 Lemma~\ref{lem:cluster} applies with $K:=S^\ominus$ and $K^\ominus=(S^\ominus)^\ominus=S$. Further if
 $(x_n)_{n=1}^\infty$ is any of the sequences from Examples~\ref{ex:DR}, \ref{ex:VN}, or \ref{ex:relaxedDR} and it admits a convergent subsequence, then as noted above, it perforce converges in norm.
\end{remark}

The following lemma shows that the sequence $(x_n)_{n=1}^\infty$ converges in norm under additional `compatibility' assumptions.

\begin{lemma}[Norm convergence]\label{lem:posProjNorm}
 Let $\H$ be a Hilbert lattice with lattice cone $S:=\H^+$, and $\kappa_{n+1}:=-x_n^-$.  Suppose $(\lambda_n)_{n=1}^\infty\subseteq\Lambda$
 for some set $\Lambda$ such that $Q(\Lambda)\subseteq S\cup(-S)$, and one of
  \begin{enumerate}[(a)]
   \item $Q\lambda_{n+1}\in S$ whenever $x_n\in S$,
   \item  If $x_{n_0}\in S$ for some $n_0$ then $(Q\lambda_n)_{n=1}^\infty$ is eventually zero,
  \end{enumerate}
 holds. Then $(x_n^+)_{n=1}^\infty$ converges in norm as soon as  $(x_n^+)_{n=1}^\infty$ remains bounded.
\end{lemma}

\begin{proof}
 In this setting equation \eqref{eq:iterSch} becomes
  \begin{equation}\label{eq:l2}
    x_{n+1}=x_n^++Q\lambda_{n+1}.
  \end{equation}
 We consider the two possible cases: (i) $Q\lambda_n\in (-S)$ for all $n\geq 1$, or  (ii) $Q\lambda_{n_0}\in S$ for some $n_0\geq 1$.

 (i) For all $n\geq 1$,
  $$x_{n+1}^+=(x_n^++Q\lambda_{n+1})^+\leq x_n^++(Q\lambda_{n+1})^+=x_n^+.$$
  Since the sequence $(x_n^+)_{n=1}^\infty$ is bounded and decreasing, by Fact \ref{fact:daniell} (a) it converges in norm.

 (ii) By  \eqref{eq:l2}  $Q\lambda_{n_0}\in S$ implies that $x_{n_0}\in S$. So, if (a) holds we have $Q\lambda_{n_0+1}\in S$ and inductively $x_n$
 and $Q\lambda_n\in S$ for $n\geq n_0$. In which case, for $n\geq n_0$,
 $$x_{n+1}^+ =x_{n+1}=x_n^++Q\lambda_{n+1}\geq x_n^+.$$
 So, $(x_n^+)_{n=n_0}^\infty$  is increasing, and by assumption bounded, hence norm convergent by Fact \ref{fact:daniell}(a).

 On the other hand, if (b)  holds then there exists $k_0$ such that $Q\lambda_n=0$ for all $n\geq k_0$. This implies that  $(x_n)_{n=1}^\infty$ is positive and constant from $n=k_0-1$ onwards.  A fortiori, $(x_n^+)_{n=1}^\infty$ converges in norm.
\end{proof}

\begin{remark}
 Condition (b) of Lemma~\ref{lem:posProjNorm} is satisfied, for example,  by the von Neumann sequence of Example~\ref{ex:VN}, and under an additional assumption, by the Douglas--Rachford sequence of Example~\ref{ex:DR}
\end{remark}

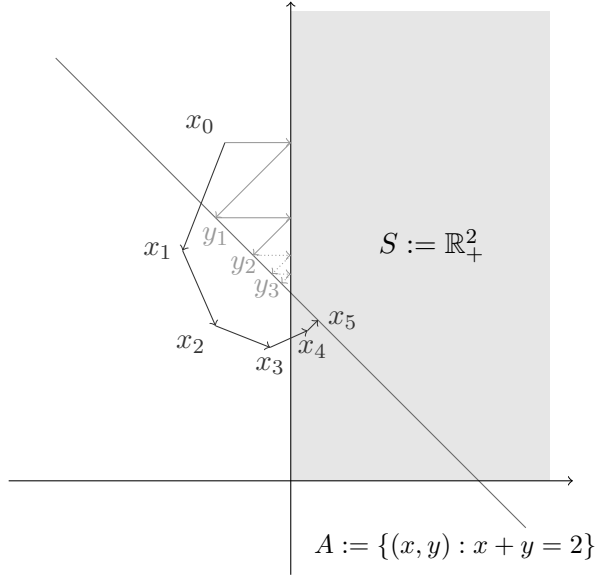
\begin{figure}
 \begin{center}\begin{tikzpicture}[scale=1.25]
  \fill[black!10] (0,0) -- (0,5) -- (2.75,5) -- (2.75,0);
  \draw[black!60] (-2.5,4.5) -- (2.5,-0.5);
  \draw (1.5,2.5) node {$S:=\mathbb R^2_+$};
  \draw (1.75,-0.7) node {\small $A:=\{(x,y):x+y=2\}$};
  \draw[->] (-3,0) -- (3,0);
  \draw[->] (0,-1) -- (0,5.1);
  \draw[->,black!80] (-0.7,3.6) node[above left] {$x_0$} -- (-1.15,2.45);
  \draw[->,black!80] (-1.15,2.45) node[left] {$x_1$} -- (-0.8,1.65);
  \draw[->,black!80] (-0.8,1.65) node[below left] {$x_2$} -- (-0.22,1.42);
  \draw[->,black!80] (-0.22,1.42) node[below] {$x_3$} -- (0.18,1.6);
  \draw[->,black!80] (0.18,1.6) node[below] {{ }$x_4$} -- (0.29,1.71) node[right] {$x_5$};
  \draw[->,black!40] (-0.7,3.6) -- (0,3.6);
  \draw[->,black!40] (0,3.6) -- (-0.8,2.8);
  \draw[->,black!40] (-0.8,2.8) node[below] {$y_1$} -- (0,2.8);
  \draw[->,black!40] (0,2.8) -- (-0.4,2.4);
  \draw[densely dotted,->,black!40] (-0.4,2.4)  -- (0,2.4);
  \draw[black!40] (-0.5,2.5) node[below] {$y_2$};
  \draw[densely dotted,->,black!40] (0,2.4) -- (-0.2,2.2);
  \draw[densely dotted,->,black!40] (-0.2,2.2)  -- (0,2.2);
  \draw[black!40] (-0.25,2.25) node[below] {$y_3$};
  \draw[densely dotted,->,black!40] (0,2.2) -- (-0.1,2.1);
 \end{tikzpicture}\end{center}
 \caption{A Douglas--Rachford sequence $(x_n)_{n=1}^\infty$ converges in five iterations, as described in Lemma~\ref{lem:DRpoint}. For the same initial point, the von Neumann sequence $(y_1)_{n=1}^\infty$ does not terminate finitely.}\label{fig:DRfinite}
\end{figure}

 Consider the sequence $(Q\lambda_n)_{n=1}^\infty$ in the von Neumann sequence of Example~\ref{ex:VN}.
 A useful observation of \cite{BaB93} is that as $x_n\in A$, one has
  $$Q\lambda_{n+1}=Q(x_n-P_Sx_n)=Q(P_{S^\ominus}x_n).$$
 Hence if $x_{n_0}\in S$ then $P_{S^\ominus}x_{n_0}=0$, so $Q\lambda_{n_0+1}=0$. Thus, inductively we see
 that condition (b) of Lemma~\ref{lem:posProjNorm} is satisfied provided $x_{n_0}\in S$ for some $n_0$.
 In which case the Von Neumann sequence is eventually constant.

 For the Douglas--Rachford sequence it is not as straightforward to select a point in $A$. Nevertheless, a similar argument can be performed, under an additional assumption, using the point given in the following lemma (see also, Figure~\ref{fig:DRfinite}).

\begin{lemma} \label{lem:DRpoint}
 Let $(x_n)_{n=1}^\infty$ be the Douglas--Rachford sequence defined by $x_{n+1}:=T_{S,A}x_n$ where $S\subseteq\H$
 is a nonempty closed convex cone, and $A\subseteq\H$ is a closed affine subspace with finite codimension.
 (i.e., $\kappa_{n+1}:=P_{S^\ominus}x_n$ and $\lambda_{n+1}:=\overline x_n-R_Sx_n$ where $\overline x_n\in A$). Then
  $$x_{n+1}-P_{S^\ominus}x_n\in A.$$
 Furthermore, if $Q(-S^\ominus)\subseteq S$ and $x_{n_0}\in S$, for some $n_0\geq 1$, then $Q\lambda_{k}=0$ for all $k\geq {n_0}+1$, and hence the Douglas--Rachford sequence is eventually constant.
\end{lemma}
\begin{proof}
 Apply $Q$ to both sides of \eqref{eq:DRiter} and use Theorem~\ref{th:moreau} to obtain
  $$Q(x_{n+1}-P_{S^\ominus}x_n)=Q\overline x_n\text{ for }\overline x_n\in A.$$
 Suppose further that $Q(-S^\ominus)\subseteq S$ and $x_{n_0}\in S$ for some $n_0\geq 1$. Then $R_Sx_{n_0}=x_{n_0}$ and $x_{n_0+1}=P_Ax_{n_0}\in A$. Since $x_{n_0}-P_{S^\ominus}x_{n_0-1}\in A$, we have
  $$Q\lambda_{n_0+1}=Q((x_{n_0}-P_{S^\ominus}x_{n_0-1})-x_{n_0})=Q(-P_{S^\ominus}x_{n_0-1})\in S,$$
 and therefore $x_{n_0+1}=x_{n_0}+Q\lambda_{n_0+1}\in S$. That is, $x_{n_0+1}\in S\cap A\subseteq\Fix T_{S,A}$ and $Q\lambda_k=0$ for all $k\geq n_0+1$.   

\end{proof}

\begin{lemma}[Iteration for a hyperplane]\label{lem:DRcodim1}
 Let $\H$ be a Hilbert lattice with Hilbert cone $S:=\mathcal H^+$, $\kappa_{n+1}:=-x_n^-$, and $Q$ be the projection onto a $1$-dimensional subspace. If $(x_n)_{n=1}^\infty$ is bounded. and $(x_n)_{n=1}^\infty$ fails to converge in norm, then $Q(S)\subseteq S$ and $R(Q)\subseteq S\cup(-S)$.
\end{lemma}
\begin{proof}
 Since the range of $Q$ has dimension $1$, we may write $Q=\langle a,\cdot\rangle a$ for some $a$ with $\|a\|=1$. Since, for any $\alpha_n$ with $Q\alpha_n\not=0$,
  $$\frac{\langle a,\alpha_n\rangle a}{\|\langle a,\alpha_n\rangle a\|}=\frac{\langle a,\alpha_n\rangle a}{|\langle a,\alpha_n\rangle|}\in\{\pm a\},$$
we see that the only possible cluster points of $(Q\lambda_n/\|Q\lambda_n\|)_{n\in\{n\in\mathbb N:Q\lambda_n\neq0\}}$ are $\pm a$.
 Hence, by Lemma~\ref{lem:cluster}, if $(x_n)_{n=1}^\infty$ fails to converge in norm then $a\in S\cup(-S)$. Since $\H$ is a Hilbert lattice, it follows that $Q(S)\subseteq S$ and that, for any $x\in\H$, $Qx=\langle a,x\rangle a\in\mathbb{R}a\subseteq S\cup(-S)$.
\end{proof}

We now specialize our results to projection/reflection methods.

\begin{theorem}[Norm convergence of Douglas--Rachford sequences]\label{th:DRcvgt}
 Let $\mathcal H$ be a Hilbert lattice, $S:=\H^+$, let $A$ be a closed affine subspace with finite codimension, and suppose  $S\cap A\neq\emptyset$. For any $x_0\in\H$
 define $x_{n+1}:=T_{S,A}x_n$. Then $(x_n)_{n=1}^\infty$ converges in norm to a point $x$ with $x^+\in S\cap A$ whenever one of the following conditions holds:
  \begin{enumerate}[(a)]
   \item $R(Q) \cap  S=\{0\}$.
   \item $Q(A-S)\subseteq S\cup(-S)$ and $Q(S)\subseteq S$. 
   \item $A$ has codimension $1$.
  \end{enumerate}
\end{theorem}

\begin{proof} (a) Follows directly from Lemma \ref{lem:cluster}. (b) By the definition of the Douglas--Rachford sequence, we have $(\lambda_n)_{n=1}^\infty\subseteq \Lambda:=A-S$. By Lemma~\ref{lem:DRpoint}, we may express
  $$\lambda_{n+1}=(x_n+x_{n-1}^-)-|x_n|=x_{n-1}^--2x_n^-.$$
Thus if $x_{n}\in S$, for some $n\geq 1$, then $Q\lambda_{n+1}=Qx_{n-1}^-\in Q(S)\subseteq S$. We therefore have that Lemma~\ref{lem:posProjNorm}(a) holds, and thus that $(x_n^+)_{n=1}^\infty$ converges in norm.   Since $(x_n)_{n=1}^\infty$ is bounded (being weakly convergent) from \eqref{eq:DRiter}
 we see that $(Q\lambda_n)_{n=1}^\infty$ is also bounded. As it is contained in a finite dimensional subspace, it contains a
 norm convergent subsequence $(Q\lambda_{n_k})_{k=1}^\infty$. Again by \eqref{eq:DRiter} we see that $(x_{n_k})_{k=1}^\infty$ converges.
 Fej\'er monotonicity now implies norm convergence.
 (c) If $(x_n)_{n=1}^\infty$ fails to converge in norm, then Lemma~\ref{lem:DRcodim1} implies  $Q(S)\subseteq S$ and $Q(A-S)\subseteq S\cup(-S)$. But then (b) implies that $(x_n)_{n=1}^\infty$ was actually norm convergent, which is a contradiction.
\end{proof}

Within this framework, we  also recover the the corresponding results relating to von Neumann sequences originally derived in \cite{BaB93}.

\begin{theorem}[Norm convergence of von Neumann sequences]\label{th:VNcvgt}
 Let $\mathcal H$ be a Hilbert lattice, $S:=\H^+$, $A$ an affine subspace with finite codimension, and $S\cap A\neq\emptyset$. For any $x_0\in\H$ define $x_{n+1}:=P_AP_Sx_n$. Then $(x_n)_{n=1}^\infty$ converges in norm to a point $x\in A\cap S$ whenever one of the following conditions holds:
  \begin{enumerate}[(a)]
    \item $R(Q) \cap  S=\{0\}$.
   \item $Q(S)\subseteq S\cup(-S)$.
   \item $A$ has codimension $1$.
  \end{enumerate}
\end{theorem}
\begin{proof} (a) Follows directly from Lemma~\ref{lem:cluster}. (b) By the remarks preceding Lemma~\ref{lem:DRpoint}, we see that $(\lambda_n)_{n=1}^\infty\subseteq\Lambda:=S$, Lemma~\ref{lem:posProjNorm}(b) holds, and thus that $(x_n)_{n=1}^\infty$ converges in norm. (c) If $(x_n)_{n=1}^\infty$ fails to converge in norm, then Lemma~\ref{lem:DRcodim1} implies $Q(S)\subseteq S$. But then (b) implies that $(x_n)_{n=1}^\infty$ was actually norm convergent, which is a contradiction.
\end{proof}

The following lemma is an analogue of Lemma~\ref{lem:DRpoint} for the relaxed Douglas--Rachford sequences.

\begin{lemma}
 Let $(x_n)_{n=1}^\infty$ be the relaxed Douglas--Rachford sequence defined by $x_{n+1}:=T^c_{S,A}x_n$ where $S\subseteq\H$ is a
 nonempty closed convex cone, and $A\subseteq\H$ is an affine subspace with finite codimension. That is,
  $$\kappa_{n+1}=2(1-a)(1-c)P_{S^\ominus}x_n, \quad \lambda_{n+1}=2(1-b)(1-c)(\overline x_n-R_S^a x_n),$$ where $\overline x_n$ is some point selected from $A$.
   Then
   $$\frac{x_{n+1}+\left(\tau-1\right)x_n-(1-a)\tau P_{S^\ominus}x_n}{\tau}\in A,$$
   where $\tau:=2(1-b)(1-c)$.
 Further suppose that $(1-a)(1-c)=(1-b)(1-c)=1/2$ and $Q(-S^\ominus)\subseteq S$. If  $x_{n_0}\in S$, for some $n_0\geq 1$, then $Q\lambda_k=0$ for all $k\geq {n_0}+1$, and hence the relaxed Douglas--Rachford sequence is eventually constant.
\end{lemma}
\begin{proof}
 The proof is similar to Lemma~\ref{lem:DRpoint}. To prove the first claim, apply $Q$ to both sides of \eqref{eq:relaxDRiter} use Theorem~\ref{th:moreau}.
 
 Suppose further that $(1-a)(1-c)=(1-b)(1-c)=1/2$ and $Q(-S^\ominus)\subseteq S$. In particular, we have $\tau=1$ so that 
  $$x_{n+1}=P_Sx_n+Q(\overline x_n-R^a_Sx_n),$$ 
  and the expression for the point in $A$ reduces to $x_{n+1}-(1-a)P_{S^\ominus}x_n\in A$.
  
If $x_{n_0}\in S$ then $R_Sx_{n_0}=x_{n_0}$, $x_{n_0+1}=P_Ax_{n_0}\in A$, and  
 $$Q\lambda_{n_0+1} = Q\left((x_{n_0}-(1-a)P_{S^\ominus}x_{n_0-1})-x_{n_0}\right)=Q(-(1-a)P_{S^\ominus}x_{n_0-1})\in S.$$
As before, this implies $x_{n_0+1}=x_{n_0}+Q\lambda_{n_0+1}\in S$. That is, $x_{n_0}+1\in S\cap A\subseteq\Fix T_{S,A}$ and so $Q\lambda_k=0$ for $k\geq n_0+1$.
\end{proof}

The following result simultaneously generalizes Theorem~\ref{th:DRcvgt} and  Theorem~\ref{th:VNcvgt} to a one-parameter family of relaxed Douglas--Rachford sequences. When $a=b=0$ and $c=1/2$ we recover the Douglas-Rachford iteration, and when $a=b=1/2$ and $c=0$ we obtain the von Neumann iteration.

\begin{theorem}[Norm convergence of relaxed Douglas--Rachford sequences]\label{th:relaxedDR}
  Let $\H$ be a Hilbert lattice, $S:=\H^+$, let $A$ be a closed affine subspace with finite codimension, and suppose $A\cap S\neq\emptyset$. For any $x_0\in H$, define $x_{n+1}:=T^c_{S,A}x_n$. Then $(x_n)_{n=1}^\infty$ converges in norm whenever one of the following conditions holds:
   \begin{enumerate}[(a)]
    \item $R(Q)\cap S=\{0\}$.
    \item $Q(A-S)\subseteq S\cup(-S)$, $Q(S)\subseteq S$, $(1-a)(1-c)=(1-b)(1-c)=1/2$  and $a\in [0,1/2]$.
    \item $A$ has codimension $1$, and $(1-a)(1-c)=(1-b)(1-c)=1/2$.
   \end{enumerate}
\end{theorem}

\begin{proof}
 (a) Follows immediately from Lemma~\ref{lem:cluster}. (b) Since $0\leq a\leq 1/2$ and
  $$R_S^ax_n=ax_n+(1-a)|x_n|=x_n^++(1-2a)x_n^-,$$
 we have $R_S^ax_n\in S$ and $\lambda_{n+1}\in \Lambda:=A-S$ for all $n\geq 1$.
 By Lemma~\ref{lem:DRpoint}, we express
  $$\lambda_{n+1}=(x_n+x_{n-1}^-)-(x_n^++(1-2a)x_n^-)=x_{n-1}^--2(1-a)x_n^-.$$
 Thus if $x_n\in S$, for some $n\geq 1$, then $Q\lambda_{n+1}=Qx_{n-1}^-\in Q(S)\subseteq S$. We therefore have that Lemma~\ref{lem:posProjNorm}(a) holds, and thus $(x_n^+)_{n=1}^\infty$ converges in norm. Arguing as before, since $(x_n)_{n=1}^\infty$ is bounded (being weakly convergent) from \eqref{eq:DRiter}
 we see that $(Q\lambda_n)_{n=1}^\infty$ is also bounded. As it is contained in a finite dimensional subspace, it contains a
 norm convergent subsequence $(Q\lambda_{n_k})_{k=1}^\infty$. Again by \eqref{eq:DRiter} we see that $(x_{n_k})_{k=1}^\infty$ converges.
 Fej\'er monotonicity now implies norm convergence.  (c) If $(x_n)_{n=1}^\infty$ fails to converge in norm, then Lemma~\ref{lem:DRcodim1} implies $Q(S)\subseteq S$ and $Q(A-S)\subseteq S\cup(-S)$. But then (b) implies that $(x_n)_{n=1}^\infty$ was actually norm convergent, which is a contradiction.
\end{proof}

\begin{remark}
 One may interpret the conditions (b) and (c) of Theorem~\ref{th:relaxedDR} as follows. If $c$ (resp. $a$ and $b$) is increased, then $a$ and $b$ (resp. $c$) must decrease.
\end{remark}

\begin{remark}[Inequality constraints]
 In a Hilbert lattice, suppose that the affine constraint $A$ is replaced with the half-space constraint
 $$A':=\{x\in\H:\langle a,x\rangle\leq b\}.$$
 That is, we consider the problem of finding a point in $A'\cap S$ where $S:=\H^+$.

 We reformulated $A'$ as an equality constrained problem in $\mathcal H\times\mathbb R$ by introducing a slack variable. That is, we have the sets
 $$\widehat{A} :=\{(x,y)\in\H\times\mathbb R:\langle a,x\rangle_\H+y=b\},\quad \widehat{S}:=\H^+\times\mathbb R_+.$$
 One may now consider the problem of finding a point in $\widehat A\cap\widehat S$.
\end{remark}

The following equivalence applies to case (a) of Theorems~\ref{th:DRcvgt} and \ref{th:VNcvgt}, and shows that its hypothesis coincides with \emph{bounded linear regularity} of $(S,A)$ (see Section~\ref{sec:rates} and \cite[Th.~5.3]{BaB93}), as we describe below in Corollary~\ref{cor:blr}.

\begin{theorem}
 Suppose $T:\H\to\H$ is a linear mapping with finite rank, and $K\subseteq\H$ a convex cone. Then
  $$N(T)+K=\H\iff N(T)^\perp\cap K^\ominus=\{0\}.$$
\end{theorem}
\begin{proof}
 (``$\Longrightarrow$") Clearly, $0\in N(T)^\perp\cap K^\ominus$. Suppose there exists a non-zero $z\in N(T)^\perp\cap K^\ominus$. Then
  $$\langle z,\H\rangle=\langle z,N(T)\rangle+\langle z, K\rangle\leq 0.$$
  In particular, since $z\in\H$ we have $\|z\|\leq 0$, and hence $z=0$.

  (``$\Longleftarrow$") Suppose $N(T)^\perp\cap K^\ominus=\{0\}$. Then
   $$\overline{N(T)+K}=(N(T)+K)^{\ominus\ominus}=(N(T)^\perp\cap K^\ominus)^\ominus=\{0\}^\ominus=\H.$$
  Thus $N(T)+K$ is a convex cone which is norm dense in $\H$, and hence $T(K)$ is a convex cone which is norm dense in $R(T)$. Further since $T$ has finite rank, $R(T)$ is a Euclidean space. Since the only dense convex cone in a finite dimensional space is the entire space (see, for example, \cite[p.~269]{AB07}), $T(K)=R(T)$. Whence
   $$K+N(T)=T^{-1}T(K)=T^{-1}R(T)=\H.$$
 This completes the proof.
\end{proof}

\begin{corollary}\label{cor:blr}
 Let $\mathcal H$ be a Hilbert lattice with lattice cone $S:=\mathcal H^+$, and $Q$ be a projection onto a finite dimensional subspace. Then
 $$N(Q)+S=\mathcal H\iff R(Q)\cap S=\{0\}.$$
\end{corollary}
\begin{proof}
 Since $R(Q)$ it is a closed subspace, $-R(Q)=R(Q)=N(Q)^\perp$. Since $S$ is the Hilbert lattice cone, $S^\ominus=-S$. Altogether,
   $$N(Q)^\perp\cap S^\ominus=\{0\}\iff R(Q)\cap S=\{0\}.$$
 The result now follows from the previous Theorem.
\end{proof}

\section{Rate of Convergence}\label{sec:rates}
In this section we gives various examples and counter-examples regarding convergence rates of projection and reflection algorithms.

Recall that a pair $(A,B)$ of closed convex sets with nonempty intersection, are \emph{boundedly linearly regular} if for each bounded set $C\subseteq\H$, there exists $\kappa>0$ such that for all $x\in C$,
 $$\max\{d(x,A),d(x,B)\}\leq\kappa d(x,A\cap B).$$
For a pair of cones, $(A,B)$, the formally weaker notations of \emph{regularity} and \emph{linearly regularity}, as defined in \cite{BaB93}, coincide with bounded linear regularity (see \cite[Th.~3.17]{BaB93}).

The following example shows that even for a hyperplane, when the transversality condition
$N(Q)+S=\mathcal{H}$,
of Corollary~\ref{cor:blr}, fails, the alternating projection method need not have a uniform linear rate of convergence in any neighbourhood of the intersection.

\begin{ex}[Failure of (bounded) linear regular for the hyperplane]\label{ex:notLinReg}
 Consider the Hilbert lattice $\mathcal H:=\ell^2(\mathbb N)$ with lattice cone $\mathcal H^+ := \{x\in\mathcal H:x_k\geq 0\text{ for }k\in\mathbb N\}$, and the constraint sets
  $$A:=\{x\in\mathcal H:\langle a,x\rangle=0\},\quad S:=\mathcal H^+,$$
where $a\in \mathcal H^+$, $\|a\|=1$ and $a_m\in]0,1[$ for all $m\in
\mathbb N$.

For any initial point $x_0\in\mathcal H$, consider the \emph{von Neumann sequence} given by
 $$x_{n+1}:=P_AP_Sx_n=x_n^+-a\langle a,x_n^+\rangle.$$
Since $A\cap S=\{0\}$, and $A$ has codimension 1, Theorem~\ref{th:DRcvgt} implies that $x_n\to 0$ in norm.

For any fixed $m\in \mathbb N$, choose $\alpha_0>0$ and recursively
define
 $$\alpha_{n+1}:=(1-a_{m}^2)\alpha_n,\quad \beta_n:=\alpha_na_m.$$
 Let $x_0:=\alpha_0e_m-\beta_0 a=\alpha_0(e_{m}-a_{m}a)\in A$.

We show that the formulae $x_n=\alpha_n e_{m}-\beta_n a$ holds for all $n$. We proceed by induction on $n$. Observe that $$x_n^+=(\alpha_n-\beta_na_{m})e_{m}=\alpha_n(1-a_m^2)e_m=\alpha_{n+1}e_m.$$
Hence $\langle a,x_n^+\rangle=\alpha_{n+1}a_m=\beta_{n+1}$, and thus $x_{n+1}=\alpha_{n+1}e_m-\beta_{n+1}a$.
We have now shown that
 $$x_n=\alpha_0(1-a_m^2)^n(e_m-a_ma).$$

For each initial point (choice of $\alpha_0$ and $m$) we see that
the iterates converge linearly to $0\in S\cap A$. However, by
choosing $\alpha_0$ sufficiently small and $m$ large enough so
$1-a_m^2)$ is as near to $1$ as we please, we see that there is no
uniform linear rate of convergence over all initial points in any
neighborhood of the solution.
\end{ex}

By contrast, we now show the same problem is often solved by the Douglas--Rachford  method in finitely many steps. We note  in the Euclidean case that we always have finite convergence, as shown in Figure \ref{fig:DRfinite}, if the iteration converges to a point in the interior of $K$.

\begin{ex}[Douglas--Rachford sequences for Example~\ref{ex:notLinReg}]
 For $A$, $S$, $a$ and $x_0=\alpha_0(e_{m}-a_{m}a)$ as in Example~\ref{ex:notLinReg} we consider the Douglas--Rachford sequence
 $$x_{n+1}:=T_{S,A}x_n=x_n^+-a\langle a,|x_n|\rangle.$$
If $x_n$ has the form $x_n=\alpha_ne_m-\beta_na$ with $\alpha_n,\
\beta_n >0$ and  $\alpha_n-\beta_na_m>0$, then
$x_{n+1}=\alpha_{n+1}e_m-\beta_{n+1}a$ where
 \begin{equation}\label{eq:recurrence}
  \alpha_{n+1}:=\alpha_n-a_m\beta_n ,\quad \beta_{n+1}:=a_m\alpha_n+(1-2a_m^2)\beta_n .
 \end{equation}
Thus $\beta_{n+1}>0$ provided $m$ is chosen sufficiently large to
ensure $a_m^2 < 1/2$ and so $(\alpha_n)$ is strictly decreasing.
However, there is no guarantee that $\alpha_{n+1}$ remains positive.
Indeed, we show that $\alpha_{n_0} \leq 0$ for some (smallest)
$n_0\in \mathbb N$, in which case $x_{n_0+1} = x_{n_0+1}^+ = 0$ at
which point the Douglas-Rachford sequence terminates.

Suppose by way of a contradiction, that $\alpha_n>0$ for all $n$.
Note that this implies also that $\beta_n>0$ for all $n$.

Eliminating $(\beta_n)_{n=1}^\infty$ from \eqref{eq:recurrence} and
rearranging gives the two-term recurrence
\begin{equation}\label{eq:arec}\alpha_{n+2}=2(1-a_m^2)\alpha_{n+1}-(1-a_m^2)\alpha_n,\end{equation}
from which we deduce that the generating function for
$(\alpha_n)_{n=1}^\infty$,
 valid for $|z| \le 1$, is
  \begin{eqnarray}\label{eq:gf}
 g(z) := \sum_{n=0}^\infty\alpha_{n}\left(x\right)z^{n} = \frac {\alpha_0 + (\alpha_1- 2x\alpha_{0})z}{1-2\,zx+xz^2}=\alpha_0\frac {1-xz}{1-2\,zx+xz^2}\end{eqnarray}
where $x:= 1-a_m^2$, on noting that $\alpha_1=x\alpha_0$.
Hence
\begin{eqnarray}\label{eq:gdf}
 g'(1) = \sum_{n=1}^\infty n\alpha_{n}\left(x\right) =-\alpha_0{\frac {x \left(1-x\right) }{ \left( 1-x\right) ^{2}}}<0.
\end{eqnarray}
This shows that at least one $\alpha_n$ is strictly negative, a contradiction.
\end{ex}

\begin{remark} We may solve
\eqref{eq:arec} to show $\alpha_n = C(\sqrt{x})^n\cos(n\theta
+\phi)$ where $\theta :=\arccos\sqrt{x} \approx \pi/2-\sqrt{2(1-a_m)}s$ and so deduce that
$\alpha_n$ `typically' exhibits oscillatory behaviour around zero. That is, the solution is a superposition of scaled Chebyshev polynomials.
  We conclude that for  sufficiently large $m$, the iteration  always terminates  finitely.

   The following \emph{Maple 16} code
\begin{verbatim}
with(gfun):
    DR:=rectoproc({z(n+1)=2*x*z(n)-x*z(n-1),z(0)=alpha[0],z(1)=alpha[1]},z(n)):
    guessgf([seq((DR(m)),m=0..10)],z)[1]:latex(%);
\end{verbatim}
produces the requisite ordinary generating function
$$-{\frac { \left( 2\,x\alpha_{{0}}-\alpha_{{1}} \right) z-\alpha_{{0}}}
{x{z}^{2}-2\,zx+1}}
,$$
quite painlessly.
 \end{remark}

We have not yet exhibited an example of a Douglas-Rachford iteration
for a simplicial cone and an affine subspace that does not terminate
finitely. We now remedy the situation. To do so we start with a
useful technical result.

\begin{ex}[Condition for two Douglas--Rachford sequences to agree]\label{ex:conicide}
Let  $I \neq \emptyset$ be  finite, and let  $\{s_i\in\mathcal H:i\in I\}$ be  linearly independent unit vectors in  $\H$. Let  $S\subseteq\mathcal H$  denote the simplicial cone
 $$S:=\{x\in\mathcal H:x=\sum_{i\in I} \lambda_i s_i,\lambda_j\geq0,\forall j\in I\},$$
and set
 $$\widehat S:=S-S =\operatorname{span}\{s_i:i\in I\}=\{x\in\mathcal H:x=\sum_{i\in I} \lambda_i s_i,\lambda_j\in\mathbb R,\forall j\in I\}.$$

Consider any  affine subspace  $A$   not containing the origin such that $S\cap A\neq\emptyset$. Fix $x^\ast\in S\cap A$ and set
 $\epsilon:=\inf\{\lambda_j:x^\ast=\sum_{i\in I}\lambda_is_i,\lambda_j\neq 0,j\in I\}.$
This is well-defined because $0\not\in A$, and as $A$ is closed  $\epsilon$ is strictly positive.

We claim that
 \begin{equation}\label{eq:infDRex}
  B_\epsilon(x^\ast)\cap S=B_\epsilon(x^\ast)\cap\widehat S,
 \end{equation}
from which it follows that $P_S|_{B_\epsilon(x^\ast)}=P_{\widehat S}|_{B_\epsilon(x^\ast)}$.

To prove \eqref{eq:infDRex}, suppose there exists $x\in
B_\epsilon(x^\ast)\cap\widehat S$ but $x\not\in
B_\epsilon(x^\ast)\cap S$. If $x$ is represented as $x=\sum_{i\in
I}\alpha_i s_i$ then there must exists an index $j\in I$ such that
$\alpha_j<0$ (otherwise $x$ would be in $B_\epsilon(x^\ast)\cap S$).
Then
 $$\epsilon\geq \|x^\ast-x\|=\left\|\sum_{i\in I}(\lambda_i-\alpha_i)s_i\right\|\geq \lambda_j-\alpha_j>\lambda_j.$$
But this contradicts the definition of $\epsilon$, and the claim follows.

If $x\in B_\epsilon(x^\ast)$, nonexpansivity of all the following operators implies that $$P_Sx,P_Ax,R_Sx,R_Sx, T_{S,A}x\in B_\epsilon(x^\ast).$$
Since the projections onto $S$ and $\widehat S$ coincide within this ball, we have shown that when the initial point is chosen sufficiently close
to a point in $S\cap A$, the Douglas--Rachford iteration for the sets $S$ and $A$ coincides with  that for $\widehat S$ and $A$.
\end{ex}

\begin{ex}[Infinite Douglas--Rachford sequences]\label{ex:inf}

We begin with the case of two affine subspaces and using
Example~\ref{ex:conicide} show how this can encompas the case of an
affine subspace and a cone.

 \begin{enumerate}[(a)] \item In \cite[Sec.~2.3]{BaBPW13} it is observed that the Douglas-Rachford method applied to two lines $L_1$ and $L_2$
 making an angle strictly between $0$ and $\pi/2$ produces an infinite sequence which converges at a linear rate given by the cosine of the angle.
 Thus, typically for linear subspaces the method does not terminate finitely.

 \item
  Consider a finite dimensional subspace $\widehat S$ and a closed affine subspace $A$ of the form given in Example \ref{ex:conicide}, and
  an initial point which yields a Douglas--Rachford sequence $(x_n)_{n=1}^\infty$ converging in norm to the point
  $x^\ast=\sum_{i\in I}\alpha_is_i$ which does not terminate finitely. By replacing each $s_j$ with $-s_j$ if necessary, we may assume that
  $\alpha_i\geq0 $ for all $i\in I$. We then have that
   $$x^\ast\in A\cap S.$$
For sufficiently large $n$, the Douglas--Rachford sequence for the sets $A$ and $S$ and initial point $x_n$ coincide with
Douglas--Rachford sequence applied to the sets $A$ and $\widehat S$. In particular, we may start with two lines as in (a) : for
instance $A:=\{(x,1) \colon x \in \mathbb{R}\}$ and $S:=\{(x,x)\colon x  \ge 0 \}$.

\item [] If instead, the set $A$ contained the origin  one may satisfy
the above conditions by replacing $A$ (respectively, initial point)
by the set (respectively, point) obtained by translating by the
non-zero vector $\widehat s\in \widehat S\setminus A$.
\end{enumerate}
\end{ex}

\begin{remark} To our chagrin we have not yet found an example for the Hilbert cone and an affine subspace for which the  Douglas-Rachford
iteration does not terminate finitely. The cone $S$ above is a
lattice for the subspace it spans but this is not the whole space.
However, at least in infinite dimensions it seems likely such
sequences exist.
\end{remark}

\section{Conclusion}

Our analysis shows that issues about the strength and rate of convergence for relaxed Douglas--Rachford methods are indeed subtle.
We repeat that we are still unable to resolve case (c) of our main result, Theorem \ref{th:DRcvgt}, even for codimension 2 in the von Neumann case.

We hope, nonetheless, that we have set the foundation for a resolution of the following conjecture.

\begin{conj} Let $K$ be a Hilbert lattice cone and $A$  a finite
codimension closed affine manifold. Then, for $0 \le a,b,c <1$ the relaxed Douglas--Rachford
iteration  of Theorem \ref{th:relaxedDR} converges in norm
as soon as $$(1-a)=(1-b)=\frac{1}{2(1-c)}.$$ So, in particular, the
corresponding methods of von Neumann and Douglas-Rachford always
converge in norm.
\end{conj}

\section*{Acknowledgements}
The authors wish to thank the helpful comments of the two anonymous referees.
JMB is supported, in part, by the Australia Research Council. BS is supported, in part, by the Australia Research Council. MKT is supported, in part, by an Australian Postgraduate Award.


\begin{thebibliography}{99}

\bibitem{DRcomb} F.J. Arag\'on Artacho, J.M. Borwein and M.K. Tam.
\newblock Recent results on Douglas--Rachford methods for combinatorial optimization problems.
\newblock J. Optim. Theory Appl., in press (2013). DOI: \href{http://link.springer.com/article/10.1007%2Fs10957-013-0488-0}{10.1007/s10957-013-0488-0}

\bibitem{DRmatrix} F.J. Arag\'on Artacho, J.M. Borwein and M.K. Tam.
\newblock Douglas--Rachford feasibility methods for matrix completion problems.
\newblock Preprint: \href{http://arxiv.org/abs/1308.4243}{arXiv:1308.4243} (2013).

\bibitem{AB07} C.D. Aliprantis and K.C. Border.
\newblock Infinite dimensional analysis: A hitch-hiker’s guide.
\newblock Springer (2007).

\bibitem{AT07} C.D. Aliprantis and R. Tourky.
\newblock Cones and duality.
\newblock Graduate Studies in Mathematics 84, American Mathematical Society (2007).

\bibitem{BaBPW13}
H.H. Bauschke,  J.Y. Bello Cruz, T.T.A. Nghia, H.M. Phan and X. Wang.
\newblock The rate of linear convergence of the Douglas--Rachford algorithm for subspaces is the cosine of the Friedrichs angle. 
\newblock J. Approx. Theory, in press (2014).
\newblock doi: \href{http://linkinghub.elsevier.com/retrieve/pii/S0021904514001166}{10.1016/j.jat.2014.06.002}.

\bibitem{BaB93} H.H. Bauschke and J.M. Borwein.
\newblock On the convergence of von Neumann's alternating projection algorithm for two sets.
\newblock Set-Valued Analysis, 1:185--212 (1993).

\bibitem{BaB96} H.H. Bauschke and J.M. Borwein.
\newblock On projection algorithms for solving convex feasibility problems.
\newblock SIAM Rev., 38(3):367--426 (1996).

\bibitem{BBDHV05} H.H. Bauschke, J.V. Burke, F.R. Deutsch, H.S. Hundal and J.D. Vanderwerff.
\newblock A new proximal point iteration that converges weakly but not in norm.
\newblock Proc. Amer. Math. Soc., 133:1829--1835 (2005).

\bibitem{BaC11} H.H. Bauschke and P.L. Combettes.
\newblock Convex analysis and monotone operator theory in Hilbert spaces.
\newblock Springer (2011).

\bibitem{BaC01} H.H.Bauschke and P.L. Combettes.
\newblock A weak-to-strong convergence principle for Fej\'er monotone methods in Hilbert space.
\newblock Math. Oper. Res., 26(2):248--264 (2001).

\bibitem{BaCL04} H.H. Bauschke, P.L. Combettes, and D.R. Luke.
\newblock Finding best approximation pairs relative to two closed convex sets in Hilbert space.
\newblock J. Approx. Theory, 127:178--192 (2004).

\bibitem{BLPW13} H.H. Bauschke, D.R. Luke, H.M. Phan and X. Wang.
\newblock Restricted normal cones and the method of alternating projections: theory.
\newblock J. Set-Valued Variational Anal., 21:431--473, 2013.

\bibitem{BaN14} H.H. Bauschke and D. Noll.
\newblock On the local convergence of the Douglas--Rachford algorithm.
\newblock Preprint \href{http://arxiv.org/abs/1401.6188}{arXiv:1401.6188} (2014).

\bibitem{B83} J.M. Borwein.
\newblock Adjoint process duality.
\newblock Math. Oper. Res., 8(3):403--434 (1983).

\bibitem{B82} J.M. Borwein.
\newblock Continuity and differentiability properties of convex operators.
\newblock Proc. London Math. Soc., 3(3):420--444 (1982).

\bibitem{cycDR} J.M. Borwein and M.K. Tam.
\newblock A cyclic Douglas--Rachford iteration scheme.
\newblock J. Optim. Theory. Appl. (2013). DOI: \href{http://link.springer.com/article/10.1007/s10957-013-0381-x}{10.1007/s10957-013-0381-x}

\bibitem{cycDRinfeas} J.M. Borwein and M.K. Tam
\newblock The cyclic Douglas--Rachford method for inconsistent feasibility problems.
\newblock J. Nonlinear Convex Anal., accepted March 2014. Preprint \href{http://arxiv.org/abs/1310.2195}{arXiv:1310.2195}

\bibitem{BV10} J.M. Borwein and J.D. Vanderwerff.
\newblock Convex functions: constructions, characterizations and counterexamples.
\newblock Cambridge University Press (2010).

\bibitem{BY84} J.M. Borwein and D.T. Yost.
\newblock Absolute norms on vector lattices.
\newblock Proc. Edinburgh Math. Soc., 27:215--222 (1984).

\bibitem{CR11} P.L. Combettes and N.N. Reyes.
\newblock Moreau's decomposition in Banach spaces.
\newblock Math. Program., 139(1):103--114 (2013).

\bibitem{ERT07} V. Elser, I. Rankenburg and P. Thibault.
\newblock Searching with iterated maps.
\newblock Proc. Nation Acad. Sci., 104(2):418--423 (2007).

\bibitem{ER11} R. Escalante and M. Raydan.
\newblock Alternating projection methods.
\newblock SIAM (2011).

\bibitem{GE08} S. Gravel and V. Elser.
\newblock Divide and concur: A general approach to constraint satisfaction.
\newblock Phys. Rev. E, 78(3):036706 (2008).

\bibitem{HL13} R. Hesse and D.R. Luke.
\newblock Nonconvex notions of regularity and convergence of fundamental algorithms for feasibility problems.
\newblock SIAM J. Optim., 23(4):2397--2419 (2013).

\bibitem{HLN13} R. Hesse, D.R. Luke and P. Neumann.
\newblock Projection methods for sparse affine feasibility: results and counterexamples.
\newblock Preprint \href{http://arxiv.org/abs/1212.3349}{arXiv:1212.3349} (2013).

\bibitem{HH04} H.S. Hundal.
\newblock An alternating projection that does not converge in norm.
\newblock Nonlinear Anal.: Theory, Methods \& Appl., 57(1):35--61 (2004).

\bibitem{K08} E. Kopeck{\'a}.
\newblock Spokes, mirros and alternating projections.
\newblock Nonlinear Anal.: Theory, Methods \& Appl., 68(6):1759--1764 (2008).

\bibitem{LLM09} A.S. Lewis, D.R. Luke and J. Malick.
\newblock Local linear convergence for alternating and averaged nonconvex projections.
\newblock Found. Comput. Math. 9:485--513, 2009.

\bibitem{LM67} P.-L. Lions and B. Mercier.
\newblock Splitting algorithms for the sum of two nonlinear operators.
\newblock SIAM J. Numer. Anal., 16:964--979 (1979).

\bibitem{LT} J. Lindenstrauss and L. Tzafriri.
Classical Banach Spaces I and II.
\newblock Springer (1996).

\bibitem{MR03} E. Matou{\v{s}}kov{\'a} and S. Reich,
\newblock The Hundal example revisited.
\newblock J. Nonlinear Convex Anal. 4:411--427 (2003).

\bibitem{P67} A.L. Peressini.
Ordered topological vector spaces.
\newblock Harper \& Row (1967).

\end{thebibliography}
\end{document}